\documentclass[12pt]{article}
\title{Updating the error term in the prime number theorem}
\author{Tim Trudgian\footnote{Supported by Australian Research Council DECRA Grant DE120100173.}\\ Mathematical Sciences Institute\\ The Australian National University, ACT 0200, Australia\\ timothy.trudgian@anu.edu.au }

\usepackage{url}
\usepackage{amsthm}
\usepackage{amsmath}
\usepackage{amssymb}
\usepackage{booktabs}
\newtheorem{thm}{Theorem}
\newtheorem{ex}{Example}
\newtheorem{cor}{Corollary}
\newtheorem{lem}{Lemma}

\begin{document}

\maketitle

\begin{abstract}
\noindent An improved estimate is given for $|\theta(x) -x|$, where $\theta(x) = \sum_{p\leq x} \log p$. Four applications are given: the first to arithmetic progressions that have points in common, the second to primes in short intervals,  the third to a conjecture by Pomerance, and the fourth to an inequality studied by Ramanujan.
Dedicated to MG Johnson, RJ Harris, PM Siddle, and NM Lyon, all of whom enabled me to work two extra days on this article.
\end{abstract}

\section{Introduction}
One version of the prime number theorem is that $\theta(x) \sim x$, where $\theta(x) = \sum_{p\leq x} \log p$. Several applications call for an explicit estimate on the error $\theta(x) - x$. Schoenfeld, \cite[Thm 11]{Schoenfeld} proved that for
\begin{equation}\label{ep}
\epsilon_{0}(x) = \sqrt{\frac{8}{17 \pi}} X^{1/2}e^{-X}, \quad X = \sqrt{(\log x)/R_{0}}, \quad R_{0}= 9.6459, 
\end{equation}
the following inequality holds
$$|\theta(x) - x| \leq x \epsilon_{0}(x), \quad (x \geq 101).$$ 
The pair of numbers $(R_{0}, 17)$ in (\ref{ep}) is particularly interesting. These arise from \cite[Thm 1]{RS2}, namely, the theorem that 
\begin{equation}\label{zero}
\textrm{$\zeta(s)$ has no zeroes in the region} \; \sigma \geq 1- \frac{1}{R \log |\frac{t}{B}|}, \quad (t\geq t_{0})
\end{equation}
for $R= R_{0}$, $B= 17$ and $t_{0} = 21$. Ramar\'{e} and Rumely \cite[p.\ 409]{RamRum} proved (\ref{zero}) with $(R, B, t_{0}) = (R_{0}, 38.31, 1000)$; Kadiri \cite{Kadiri} proved (\ref{zero}) with $(R, B, t_{0}) = (5.69693, 1, 2)$.

A meticulous overhaul of Schoenfeld's paper would be required to furnish a `general' version of (\ref{zero}), that is, one in which $B$ and $R$ are chosen for maximal effect. This article does not attempt such an overhaul. Rather, forcing $B$ to be $17$ in (\ref{zero}) means that many of the numerical estimations in Schoenfeld's article can be let through to the keeper. With $B= 17$ one can obtain admissible values of $R$ and $t_{0}$ in (\ref{zero}) as follows.

Let the Riemann hypothesis be true up to height $H$: by Platt \cite{Plattarxiv} we have $H= 3.061\times 10^{10}$. Let $\rho$ represent a non-trivial zero of $\zeta(s)$ with $\rho = \beta + i\gamma$. Using Kadiri's result we see that
\begin{equation*}
\beta \leq 1 - \frac{1}{5.69693 \log t} \leq 1 - \frac{1}{R \log |\frac{t}{17}|}
\end{equation*}
provided that 
$$t\geq \exp\left\{\frac{R \log 17}{R - 5.69693}\right\}.$$
Set $H = \exp\{R \log 17/(R - 5.69693)\}$, whence we may take $R = 6.455$. We conclude that there are no zeroes in 
\begin{equation}\label{zero2}
\sigma \geq 1 - \frac{1}{6.455 \log|\frac{t}{17}|}, \quad (t\geq 24).
\end{equation}
This enables us to prove good bounds for $\theta(x) -x$ and for $\psi(x) -x$, where $\psi(x) = \sum_{p^{m}\leq x} \log p$, as indicated in the following theorem.
\begin{thm}\label{TSwan}
Let 
\begin{equation*}\label{ep2}
\epsilon_{0}(x) = \sqrt{\frac{8}{17 \pi}} X^{1/2}e^{-X}, \quad X = \sqrt{(\log x)/R}, \quad R = 6.455.
\end{equation*}
Then
\begin{equation*}
\begin{split}
|\theta(x) - x| &\leq x \epsilon_{0}(x), \quad (x \geq 149)\\
|\psi(x) - x| &\leq x \epsilon_{0}(x), \quad (x \geq 23).
\end{split}
\end{equation*}
\end{thm}
Throughout Schoenfeld's paper numerous bounds on $x$ are imposed, where $X = \sqrt{(\log x)/R_{0}}$. Fortunately, for our purposes, all of these arise from bounds imposed on $X$. For example, the first bound in \cite[(7.30)]{Schoenfeld} requires $X \geq 17/2\pi$. With our value of $R$ we need $\log x \geq 48$ compared with Schoenfeld's requirement  $\log x \geq 71$.
Making these slight changes throughout pp.\ 342-348 of \cite{Schoenfeld} we find that
\begin{equation}\label{in}
|\psi(x) - x|,  |\theta(x) - x| \leq x \epsilon_{0}(x), \quad (\log x \geq 1163).
\end{equation}
In order to prove Theorem \ref{TSwan} we cover small values of $x$ following the approach on pp.\ 348-349 of \cite{Schoenfeld} but using the superior bounds on $|\psi(x) -x|$ as given by Faber and Kadiri \cite{Faber}.
We make use of equation ($5.3^{*}$) in \cite{Schoenfeld}, namely, 
\begin{equation*}
\psi(x) - \theta(x) < 1.001093 x^{1/2} + 3x^{1/3}
 \leq A(x_{0}) x, \quad (x\geq x_{0})
\end{equation*}
where $A(x_{0}) = 1.001093 x_{0}^{-1/2} + 3x_{0}^{-2/3}.$
For $e^{25} \leq x \leq e^{45}$ we have, by \cite[Table~3]{Faber},
\begin{equation*}
|\psi(x) - x|, \quad |\theta(x) - x| \leq (A(e^{25}) +4.9\times 10^{-5}) \frac{x\epsilon_{0}(x)}{\epsilon_{0}(e^{45})} \leq 0.003 x\epsilon_{0}(x).
\end{equation*}
Now for $e^{45} \leq x \leq e^{1163}$ we have
\begin{equation*}
|\psi(x) - x|, \quad |\theta(x) - x| \leq (A(e^{45}) +1.1\times 10^{-8}) \frac{x\epsilon_{0}(x)}{\epsilon_{0}(e^{1162})} \leq 0.006 x\epsilon_{0}(x).
\end{equation*}
Hence (\ref{in}) is true for all $x\geq e^{25}$. For $x< e^{25}$ note that $\epsilon_{0}(x)$ increases for $X<\frac{1}{2}$ and decreases thereafter. Therefore
\begin{equation}\label{last}
\epsilon_{0}(x) \geq \min\{\epsilon_{0}(2), \epsilon_{0}(e^{25})\} \geq 0.075.
\end{equation}
Theorem 10 in \cite{RS1} gives $\theta(x) > 0.93x$ for $x\geq 599$. This, combined with (\ref{last}), shows that 
\begin{equation}\label{had}
\theta(x) -x > -x\epsilon_{0}(x), \quad (x\geq 599).
\end{equation}
Since $\psi(x) \geq \theta(x)$ the inequality in (\ref{had}) also holds with $\psi(x)$ in place of $\theta(x)$.
Using $\psi(x) \leq 1.04 x$ (see Theorem 12 in \cite{RS1}) and (\ref{last}) gives
\begin{equation*}\label{mu}
\theta(x) \leq \psi(x) \leq 1.04 x < x+ x\epsilon_{0}(x), \quad (2\leq x \leq e^{25}).
\end{equation*}
All that remains is to verify (\ref{had}) and the analogous inequality for $\psi(x)$ for values of $x\leq 599$ --- a computational dolly. 

\section{The difference $\pi(x) - \textrm{li}(x)$}
Let $\pi(x)$ denote the number of primes not exceeding $x$ and $\textrm{li}(x)$ denote the logarithmic integral, namely
$$\textrm{li}(x) = \lim_{\epsilon\rightarrow 0^{+}} \left(\int_{0}^{1- \epsilon} \frac{dt}{\log t} + \int_{1+ \epsilon}^{x} \frac{dt}{\log t}\right).$$
Concerning the difference $\pi(x) - \textrm{li}(x)$ we have
\begin{equation}\label{pill2}
|\pi(x) - \textrm{li}(x)| \leq 0.4394 \frac{x}{(\log x)^{3/4}} \exp(-\sqrt{(\log x)/9.696}), \quad (x\geq 59),
\end{equation}
due to Dusart \cite[Thm 1.12]{DusThesis}\footnote{There is also the result of Ford \cite{Ford} 
$$\pi(x) - \textrm{li}(x) = O( x \exp\{-0.2098 (\log x)^{3/5} (\log\log x)^{-1/5}\}).$$ It appears that this result has not been made explicit.}. 
Good bounds on $\pi(x) - \textrm{li}(x)$ can be obtained from good bounds on $\theta(x) -x$, since
\begin{equation}\label{pi}
\begin{split}
\pi(x) - \textrm{li}(x) &= \frac{\theta(x)}{\log x} + \int_{2}^{x}\frac{\theta(t)}{t \log^{2} t}\, dt - \int_{2}^{x} \frac{dt}{\log t} - \textrm{li}(2)\\
&=  \frac{\theta(x)-x}{\log x} +  \frac{2}{\log 2} +  \int_{2}^{x}\frac{\theta(t)-t}{t \log^{2} t}\, dt - \textrm{li}(2).
\end{split}
\end{equation}
Using Theorem \ref{TSwan} we can prove
\begin{thm}\label{TLi}
\begin{equation*}
|\pi(x) - \textrm{li}(x)| \leq 0.2795 \frac{x}{(\log x)^{3/4}} \exp\left(-\sqrt{\frac{\log x}{6.455}}\right), \quad (x\geq 229).
\end{equation*}
\end{thm}
\begin{proof}
We split up the range of integration in (\ref{pi}) so that $\int_{2}^{x} = \int_{2}^{x_{0}} + \int_{x_{0}}^{x} = I_{1} + I_{2}$ for some $x_{0}\geq 149.$ To estimate $I_{2}$ we use Theorem \ref{TSwan} and consider
$$g(t) = \frac{t \epsilon_{0}(t)}{(\log t)^{\alpha + \frac{1}{4}}}.$$
The value of $\alpha$ in the expression for $g(t)$ must be less than $7/4$. Following Dusart we choose $\alpha = 7/5$, whence it is easy to verify that $\epsilon_{0}(t)/(\log t)^{2} < g'(t)$ for all $t\geq 149$.

To estimate $I_{1}$ we invoke \cite[Thm 19]{RS1}
$$ \theta(t) < t, \quad t<10^{8}.$$
Interchanging summation and integration we have
\begin{equation*}\label{dust}
\int_{2}^{x_{0}} \frac{\theta(t)}{t\log^{2} t}\, dt = \int_{2}^{x_{0}} \frac{\sum_{p\leq t}\log p}{t \log^{2} t}\, dt =  \sum_{p\leq x_{0}} \log p \int_{p}^{x_{0}} \frac{dt}{t \log^{2} t} = \pi(x_{0}) - \frac{\theta(x_{0})}{\log x_{0}}.
\end{equation*}
Therefore (\ref{pi}) becomes
\begin{equation*}\label{kite}
\begin{split}
|\pi(x) - \textrm{li}(x)| \leq &\frac{x \epsilon_{0}(x)}{\log x} + \frac{x\epsilon_{0}(x)}{(\log x)^{33/20}} + \frac{2}{\log 2} - \textrm{li}(2)  -\frac{x_{0}\epsilon_{0}(x_{0})}{(\log x_{0})^{33/20}} \\
&+ \int_{2}^{x_{0}} \frac{dt}{\log^{2} t} -\pi(x_{0}) + \frac{\theta(x_{0})}{\log x_{0}}.
\end{split}
\end{equation*}
We may choose $x_{0}$ in (\ref{pi}) subject to $149\leq x_{0}\leq 10^{8}$. Choosing $x_{0} = 10^{8}$ shows, in less than 3 minutes  using \textit{Mathematica} on a 1.8GHz laptop, that
\begin{equation}\label{pi2}
|\pi(x) - \textrm{li}(x)| \leq \frac{x \epsilon_{0}(x)}{\log x} + \frac{x\epsilon_{0}(x)}{(\log x)^{33/20}}  \leq 1.151 \frac{x \epsilon_{0}(x)}{\log x},
\end{equation}
for $x\geq 10^{8}$. For smaller $x$ we note that, by Kotnik \cite{Kotnik} $\pi(x) < \textrm{li}(x)$ for $2\leq x \leq 10^{14}$. Therefore
\begin{equation}\label{picheck}
\max_{x\in[p_{k}, p_{k+1})} |\pi(x) - \textrm{li}(x)| = \max_{x\in[p_{k}, p_{k+1})} \textrm{li}(x) - \pi(x) \leq \textrm{li}(p_{k+1}) - k.
\end{equation}
Using (\ref{picheck}) we verify (\ref{pi2}) for all $x\geq p_{k}$ with $k \geq 48$, which is equivalent to $x\geq 229$, which proves the theorem.
\end{proof}

\section{Applications}
We now present four applications of Theorem \ref{TSwan} and \ref{TLi}. We stress that explicit results of this nature have many uses throughout the literature; our list of four applications is by no means exhaustive. One striking example of this applicability is Helfgott's proof of the ternary Goldbach conjecture \cite{Helfgott}. In \cite[\S 7]{Helfgott} Helfgott makes frequent use of estimations for the number of primes in short intervals and the size of the Chebyshev functions. 
\subsection{Intersecting arithmetic progressions} 
Let $N_{t}(k)$ denote the maximum number of distinct arithmetic progressions of $k$ numbers such that any pair of progressions has $t$ members in common. Ford \cite{FordInt} considers the following example.

\begin{ex}
For $1\leq i < j \leq k$, let $B_{ij}$ be the arithmetic progression the $i$th element of which is 0, and the $j$th element of which is $k!$.
\end{ex}

Ford shows, in Theorem 3 of \cite{FordInt}, that for all $k\geq 10^{8000}$, $N_{2}(k) = k(k-1)/2$, and that every configuration of $k(k-1)/2$ arithmetic progressions with 2 points in common is equivalent (up to translations and dilations) to the arithmetic progression in Example 1. We are able to use Theorem \ref{TSwan} to prove
\begin{cor}\label{ford}
For $k \geq 10^{4848}$ we have $N_{2}(k) = k(k-1)/2$ and that every configuration of $k(k-1)/2$ arithmetic progressions with 2 points in common is equivalent to the arithmetic progression in Example 1.
\end{cor}
\begin{proof}
Analogous to Lemmas 3.3 and 3.4 in \cite{FordInt} we can show that for $k\geq e^{280}$ there is always a prime in the interval $[k, k+a]$ where \begin{equation}\label{adef}
a = 0.56 k(\log k)^{1/4} \exp(-\sqrt{(\log k)/6.455}).
\end{equation} 
We now follow the proof of Theorem 3 in \cite{FordInt} using our (\ref{adef}) in place of his $a =  0.44 k(\log k)^{1/4} \exp(-0.321979\sqrt{\log k}).$
\end{proof}
It is worthwhile to remark that Corollary \ref{ford} could be improved if the method in \cite{Ford04} were made explicit. However, it seems unlikely that one could reduce the bound on $k$ in Corollary \ref{ford} to a height below which direct computation could be carried out.

\subsection{Primes in short intervals}
Various results have been proved about the existence of a prime in a short interval $[x, x+ f(x)]$ where $f(x) = o(x)$. For example, Dusart \cite[Prop.\ 6.8]{Dusart} has shown that there exists a prime in the interval $[x, x + \frac{x}{25\log^{2}x}]$ whenever $x\geq 396738$. We improve this in
\begin{cor}\label{primes}
For all $x\geq 2898239$ there is a prime in the interval
$$ \left[x, x\left(1 + \frac{1}{111\log^{2}x}\right)\right].$$
\end{cor}
We first prove the following
\begin{lem}\label{lemonly}
For $x\geq e^{35}$ we have
\begin{equation}\label{thmain}
|\theta(x)- x| \leq \frac{0.0045x}{\log^{2} x}.
\end{equation}
\end{lem}
\begin{proof}
Using ($5.3^{*}$) of \cite{Schoenfeld} we have
\begin{equation}\label{ee}
\frac{|\theta(x) -x|\log^{2} x}{x} \leq \left(\frac{|\psi(x) -x|}{x} \log^{2}x + \frac{1.001093 \log^{2} x}{\sqrt{x}} + \frac{3\log^{2}x}{x^{2/3}}\right)= B(x),
\end{equation}
say.
According to Table 3 in \cite{Faber}, $|\psi(x) - x|\leq 7.4457\times 10^{-7}$ for $x\geq e^{35}$, whence $B(x)$ is bounded above by
$$ \left(7.4457\times 10^{-7} \log^{2}x + \frac{1.001093 \log^{2} x}{\sqrt{x}} + \frac{3\log^{2}x}{x^{2/3}}\right)\bigg|_{x = e^{75}}, \quad x\in[e^{35}, e^{75}],$$
which is bounded above by 0.0042. We continue in this way, using intervals of the form $[e^{a}, e^{b}]$, Faber and Kadiri's bounds at $e^{a}$ and evaluating $B(x)$ at $x=e^{b}$. The results are summarised below in Table \ref{table}. 

\begin{table}[ht]
\caption{Bounding $\theta(x) - x$}
\label{table}
\centering
\begin{tabular}{lc}
\hline\hline
Interval & Bound on $B(x)$ in (\ref{ee})\\[0.5ex]\hline
$[e^{35}, e^{75}]$ & 0.0042\\
$[e^{75}, e^{1500}]$ & 0.0037\\
$[e^{1500}, e^{2000}]$ & 0.0038\\
$[e^{2000}, e^{2500}]$ & 0.0045\\
$[e^{2500}, e^{3000}]$ & 0.0044\\
$[e^{3000}, e^{4000}]$ & 0.0036\\
    \hline\hline
  \end{tabular}
\end{table}
Taking the maximum entry in the right-hand column of the table proves Lemma \ref{lemonly} for $e^{35}\leq x \leq e^{4000}$. When $x\geq e^{4000}$ we use Theorem \ref{TSwan}. This completes the proof of the lemma.
\end{proof}
Note that one could refine this result by taking more intermediate steps in the argument. For example one could use the interval $[e^{2000}, e^{2100}]$ to try to reduce the bound of $0.0045$. We have not pursued this since the entry $e^{2100}$ is not in Table 3 in \cite{Faber} and, while it could be calculated, the above lemma is sufficient for our purposes.

We now use Lemma \ref{lemonly} to exhibit primes in short intervals. Indeed, for $x\geq e^{35}$ Lemma \ref{lemonly} shows that
$$\theta\left\{x\left(1 + \frac{1}{c\log^{2} x}\right)\right\} - \theta(x)$$
is positive provided that $c\leq 111.1107\ldots$. Taking $c=111$ we conclude that there is always a prime in the interval $[x, x(1+ 1/(111 \log^{2} x))]$ whenever $x\geq e^{35}$.
This establishes Corollary \ref{primes} when $x\geq e^{35} \approx 1.58\times 10^{15}$. Rather than perform the herculean, if not impossible, feat of examining all those $x< e^{35}$ we proceed as follows.

Suppose that $p_{n+1} - p_{n}\leq X_{1}$ for all $p_{n} \leq x_{1}$, where $x_{1} \geq e^{35}$. That is, the maximal prime gap of all primes up to $x_{1}$ is at most $X_{1}$. Therefore $p_{n+1} \leq p_{n} + X_{1}$ which will be less that $p_{n}(1+ \frac{1}{111 \log^{2}p_{n}})$ as long as
\begin{equation}\label{p1}
\frac{p_{n}}{\log^{2}p_{n}} \geq 111 X_{1}.
\end{equation}
If (\ref{p1}) holds for all $y_{1} \leq p_{n} \leq x_{1}$ we can conclude that Corollary \ref{primes} holds for all $x\geq y_{1}$. If $y_{1}$ is still too high for a direct computation over all integers less than $y_{1}$, then we may play the same game again, namely: find an $x_{2}\geq y$ such that $p_{n+1} - p_{n} \leq X_{2}$.

Nyman and Nicely \cite[Table 1]{NymNic} show that one may take $x_{1} = 1.68\times 10^{15}$, which is greater than $e^{35}$, and $X_{1} = 924$. It is easy to verify that (\ref{p1}) holds for all $p_{n} \geq 3.05\times 10^{7}$.
We can now check, relatively swiftly that the maximal prime gap for $p_{n} < 3.06\times 10^{7}$ is $210$. We may now verify Corollary \ref{primes} for all $x\geq 5.63\times 10^{6}$. Two more applications of this method, using the fact that the maximal prime gap for $p_{n} < 5.7\times 10^{6}$ is $159$, and for $p_{n} < 4\times 10^{6}$ is $148$ we see that Corollary \ref{primes} is true for all $x\geq 3.8\times 10^{6}$.

We now examine $x \leq 3.8\times 10^{6}$. An exhaustive search took less than two minutes on \textit{Mathematica} --- this completes the proof of Corollary \ref{primes}.

There are several ways in which this result could be improved. Extending the work done by Nyman and Nicely \cite{NymNic} makes a negligible difference to the choice of $c$. Probably the best plan of attack is reduce the size of the coefficient in Lemma \ref{lemonly}. For example, if the coefficient in (\ref{thmain}) were reduced to $0.0039$ we could take $c=128$.

Finally, the result in Corollary \ref{primes} ought to be compared with the sharpest known result for a different short interval. Ramar\'{e} and Saouter \cite[Table 1]{RamSao} proved that there is always a prime in the interval
$$(x(1-\Delta^{-1}), x],\quad \Delta = 212215384, \quad x\geq e^{150}.$$
Corollary \ref{primes} improves on this whenever $x\geq 3.2\times 10^{600} \approx e^{1383}$. Although this value of $x$ is large by anyone's standards, it appears that Corollary \ref{primes} could be useful in searching for primes between cubes --- see \cite{DudekCubes}.

\subsection{A conjecture by Pomerance}
Consider numbers $k>1$ for which the first $\phi(k)$ primes coprime to $k$ form a reduced residue system modulo $k$. Following the lead of Hajdu, Saradha and Tijdeman \cite{HST}, we call such an integer $k$ a \textit{P-integer}. For example $12$ is a $P$-integer and $10$ is not since
$$\{5,7,11,13\} \equiv \{5, 7, 11, 1\}, \quad \quad \{3,7,11,13\} \equiv\{ 3,7,11,3\},$$
and, whereas the first is a reduced residue system, the second is not.
From \cite[Thm 2]{PomC} Pomerance deduced that there can be only finitely many $P$-integers. Hajdu, Saradha and Tijdeman \textit{[op.\ cit.]} proved, \textit{inter alia}, that if $k$ is a $P$-integer such that $k> 30$ then
$ 10^{11}< k < 10^{3500}.$
As noted by Hajdu, Saradha and Tijdeman, one may improve (\ref{pill2}) by using the zero-free region proved by Kadiri, that is, by using our Theorem \ref{TSwan}. We do this thereby proving
\begin{cor}\label{T1}
If $k$ is a $P$-integer then $k < 10^{1805}$.
\end{cor}
\begin{proof}
We use Theorem \ref{TLi} instead of Lemma 2.1(iii) in \cite{HST} and proceed as in \cite[\S 5]{HST}. Let $k\geq 10^{1805}$ and define
\begin{equation}\label{fee}
\begin{split}
f_{0}(k) &= \frac{k}{\log k/2} + \frac{k}{\log^{2} k/2} + \frac{1.8 k}{\log^{3} k/2} - \frac{k}{\log k} - \frac{k}{\log^{2}k} - \frac{2.51 k}{\log^{3} k} - \log k,\\
f_{n}(k) &= \frac{k}{4(n+1)\log^{2}(nk +k)} - 1.118 \frac{nk +k}{(\log nk)^{3/4}} \exp( - \sqrt{\log(nk)/6.455}),
\end{split}
\end{equation}
where the constant $1.118$ is four times that appearing in Theorem \ref{TLi}.
Lemma 3.1 in \cite{HST} gives the following
\begin{equation}\label{test1}
\textrm{ $k$ is not a $P$-integer if} \quad f_{0}(k) + \sum_{n=1}^{L} f_{n}(k) >0,
\end{equation}
where $L$ satisfies
$$ L \geq \frac{\log k - \log h(k)}{h(k)} -2, \quad  h(k) = 1.7811\log\log k + 2.51/(\log\log k).$$
When $k\geq 10^{1805}$ we have $L\geq 273$. We verify that the condition in (\ref{test1}) is met for $273\leq L \leq 3800$. We now proceed as in \cite[p.\ 181]{HST} with $3800$ and $k= 10^{1805}$ taking the place of $1500$ and $k= 10^{3500}$ respectively.
\end{proof}
The numbers 1.8 and 2.51 appearing in (\ref{fee}) are worth a mention. These are approximations to the number 2 that appears in the expansion
$$\pi(x) \sim \frac{x}{\log x} + \frac{x}{\log^{2} x} + \frac{2x}{\log^{3} x} + \cdots.$$
Replacing these numbers in (\ref{fee}) by 2, a situation on which one could not possibly improve, makes a negligible difference. Indeed, such a substitution could not improve the bound in Corollary 3 to $k< 10^{1803}$.

It is certainly possible that a refined version of Theorem \ref{TSwan} could resolve completely the Pomerance conjecture. Indeed, using a slightly different approach, Togb\'{e} and Yang have announced in \cite{Togbe} a proof of the conjecture.

\subsection{An equality studied by Ramanujan}
Ramanujan \cite[Ch.\ 24]{BerndtR} proved that
\begin{equation}\label{ram}
\pi(x)^{2} < \frac{e x}{\log x} \pi\left( \frac{x}{e}\right),
\end{equation}
holds for all sufficiently large values of $x$. In a paper to appear, Dudek and Platt \cite{DudekPlatt} have used Theorem \ref{TSwan} to show that, on the Riemann hypothesis, (\ref{ram}) is true for all $x> 38,358,837,682.$ It seems difficult to prove this unconditionally: in this case Dudek and Platt are able to show that (\ref{ram}) is true for all $x\geq \exp(9658)$.

\section{Conclusion}
Theorems \ref{TSwan} and \ref{TLi} could be improved in several ways. First, if one knew that the Riemann hypothesis had been verified to a height greater than $3.061\times 10^{10}$, one could reduce the coefficient in the zero-free region in (\ref{zero2}). Second, one could try to improve Kadiri's zero-free region either by reducing the value of $R$ or by improving the size of $B$ in (\ref{zero}). A higher verification of the Riemann hypothesis has a mild influence on this method of proof.

Third, one may feed any improvements in a numerical verification of the Riemann hypothesis and the zero-free region into Faber and Kadiri's argument, thereby improving the estimate on $\psi(x) - x$. Finally, one may try to overhaul completely Schoenfeld's paper in order to provide a bespoke version of Theorem~\ref{TSwan}. 


\bibliographystyle{plain}
\bibliography{themastercanada}

\end{document}